\documentclass[final]{dmtcs-episciences}

\usepackage[utf8]{inputenc}
\usepackage{subfigure}
\usepackage{graphicx}
\usepackage{amsmath, amssymb, amsthm, tikz, mathrsfs,verbatim,tensor}
\usepackage{enumerate}
\usepackage{amsfonts}
\usepackage{latexsym}
\usepackage{graphics}
\usepackage[round]{natbib}

\usepackage{bbm}

\usepackage{etoolbox}

\newtheorem{thm}{Theorem}
\newtheorem{lemma}[thm]{Lemma}

\newcommand{\ppp}[2][]{\xi_{#2}\ifthenelse{\equal{#1}{}}{}{^{#1}}}

\newcommand{\Poi}{\mathrm{Poi}}

\renewcommand{\P}{\mathbf{P}}
\newcommand{\prob}{\mathbf{P}}
\newcommand{\E}{\mathbf{E}}

\newcommand{\1}{\mathbf{1}}
\newcommand{\indc}{\1}

\newcommand{\fp}{\mathcal{F}}
\newcommand{\desc}{{desc}}
\newcommand{\word}{\mathbf{W}}
\newcommand{\nbhd}{\mathbf{N}}

\newcommand{\children}{\mathcal{C}}
\newcommand{\limtree}{\mathbf{T}}
\newcommand{\limdesc}{\mathbf{D}}

\newcommand{\bw}{{ w}}

\newcommand{\wtree}{\gamma}
\newcommand{\inert}{\square}
\newcommand{\leaves}{\mathcal{L}}
\newcommand{\cwords}{\beta}
\newcommand{\intersection}{\mathcal{B}}
\newcommand{\gw}{\mathbf{GW}}

\author{Samuel Regan\affiliationmark{1} and Erik Slivken\affiliationmark{2}\thanks{Partially supported by ERC Starting Grant 680275 MALIG}}
\title{Expected size of a tree in the fixed point forest}
\affiliation{
University of California Davis\\
Dartmouth College
}

\received{2019-3-30}
\revised{2019-7-16}
\accepted{2019-9-10}

\begin{document}
\publicationdetails{21}{2019}{2}{1}{5628}
\maketitle

\begin{abstract}We study the local limit of the fixed-point forest, a tree structure associated to a simple sorting algorithm on permutations.  This local limit can be viewed as an infinite random tree that can be constructed from a Poisson point process configuration on $[0,1]^\mathbb{N}$.  We generalize this random tree, and compute the expected size and expected number of leaves of a random rooted subtree in the generalized version.  We also obtain bounds on the variance of the size. 
\end{abstract}

\section{Introduction}

We start with a simple sorting algorithm on a deck of cards labeled $1$ though $n$.  If the value of the top card is $i$, place it in the $i$th position from the top in the deck.  Repeat until the top card is a $1$.  Viewing the deck of cards as a permutation in one-line notation $\pi=\pi(1)\pi(2)\cdots \pi(n)$, we create a new permutation, $\tau(\pi)$, by removing the \emph{value} $\pi(1)$ from beginning of the permutation and putting it into \emph{position} $\pi(1)$.  For example, if $\pi = 4 3 5 1 2$ then $\tau(\pi)=35142$.  This induces a graph whose vertices are the permutations of $[n]=\{1,\cdots,n\}$ and edges are pairs of permutations $(\pi,\tau(\pi)).$  Note that $\tau(\pi)$ has a fixed point at the position $\pi(1).$

This graph is a rooted forest, which we denote by $F_n$ and call the \emph{fixed point forest}.  A rooted forest is a union of rooted trees, and a tree is a graph that does not contain any closed loops involving distinct vertices.  A permutation that begins with 1 is called the base of the tree in which they are contained.  A thorough introduction to the fixed point forest can be found in \cite{JSS}.

The fixed point forest was first studied in \cite{McKinley}.  The largest tree in $F_n$ has size bounded between $(n-1)!$ and $e(n-1)!$ and has as its base the identity permutation.  The longest path from a leaf to a base is $2^{n-1}-1$ and is unique, starting from the permutation $2 3 \cdots n 1$ and ending at the identity.  
       
Let $\mathfrak{S}_n$ denote the set of permutations of length $n$.  For $\pi\in \mathfrak{S}_n$, let $\fp(\pi)$ denote the collection of fixed points of $\pi$ other than $1$.  For each $m\in \fp(\pi)$ we create a new permutation $\pi^{(m)}$ such that 
\begin{equation*}
	\pi^{(m)} (i ) = \left \{\begin{array}{lr} m, & i=1 \\
\pi(i-1), & 2\leq i \leq m\\
\pi(i),  & m< i \leq n \end{array} \right. .
\end{equation*}

We say we \emph{bump} the value $m$ in $\pi$ to create $\pi^{(m)}$ and call $\pi^{(m)}$ a \emph{child} of $\pi$.  We let $\children(\pi) = \{ \pi^{(m)} :  m\in \fp(\pi)\}$ denote the set of children of $\pi$.  Every child $\sigma \in \children(\pi)$ satisfies $\tau(\sigma) = \pi$ hence is connected to $\pi$ in $F_n$.  

Let $N(\pi)$ be the rooted tree in $F_n$ that contains $\pi$, with $\pi$ designated as the root instead of the unique permutation that starts with $1$ in $N(\pi)$.  Let $\desc(\pi)$ be the subtree of $N(\pi)$ rooted at $\pi$ and consisting of $\pi$ and its descendants, so that $\desc(\pi) \subseteq N(\pi).$  We call this the \emph{descendant tree} of $\pi$ (See Figure \ref{desctree}).  Note that for any permutation $\sigma \in \desc(\pi)$, there is some $r$ such that $\tau^r(\sigma) = \pi$.  


By Theorem 3.5 in \cite{JSS}, there exists a tree,  $\limtree$, such that as $n\to \infty$, for $\mathbf{\pi}_n$ chosen uniformly at random from permutations of size $n$, the randomly rooted tree $\nbhd_n = N(\mathbf{\pi}_n)$, converges in the local weak sense to $\limtree$.  This limiting tree is described in Section 2 of \cite{JSS}, and the subtree of $\limtree$ which corresponds to the local weak limit of $\desc(\mathbf{\pi}_n)$ has a similar description, denoted by $\limdesc$.  In \cite{JSS}, they find the distribution for the shortest and longest paths from the root to a leaf in $\limdesc$.  
The main purpose of the paper is to study the size of $\limdesc$.  For $\alpha\in [0,1]$, we define a generalization of $\limdesc$, denoted $\limdesc_\alpha$ such that $\limdesc = \limdesc_1$.  We compute the expected size and expected number of leaves of $\limdesc_\alpha$ and show that they are both unbounded for $\alpha =1$.  Finally we find bounds on the second moment of the size of $\limdesc_\alpha$.  We show that the second moment has a phase transition from finite to infinite somewhere between $(3-\sqrt{5})/2$ and $(\sqrt{5} -1)/2.$  

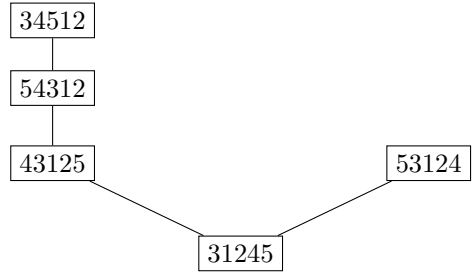
\begin{figure}\centering
\begin{tikzpicture}[
every fit/.style={ellipse,draw,inner sep=2pt},
grow'=up,
leaf/.style={draw, fill, circle, minimum size=4pt, inner sep =1pt},
level 1/.append style={sibling distance=50mm, level distance =12mm},
level 2/.append style={sibling distance=20mm, level distance =10mm},
level 3/.append style={sibling distance=10mm, level distance = 9mm},
level 4/.append style={sibling distance=5mm, level distance = 8mm},
]
\node[draw] (root){$31245$} 
	child{ node [draw] {$43125$ } 
  		child{node [draw] {$ 54312$}
  			child{node [draw]{$34512$}	
  			}
  		}
  	}
  	child{node [draw] {$53124$}
  	}
;	
\end{tikzpicture}
\caption{The descendant tree $\desc(\pi)$ for $\pi = 31245$}
	\label{desctree}
\end{figure}

\section{Local limits, point process configurations, and trees}\label{defs}

\subsection*{Poisson Point Processes}

The following briefly introduces an important probabilistic object:  Poisson point processes.  A thorough treatment can be found in \cite{Kingman}.  

We say a random variable $X$ is $\Poi(\alpha)$ if it satisfies $\prob(X = k) = \frac{1}{k!}e^{-\alpha}\alpha^k.$  If $X_0$ and $X_1$ are two independent $\Poi(\alpha_0)$ and $\Poi(\alpha_1)$, respectively, then their sum is $\Poi(\alpha_0+\alpha_1)$.  

A point process on $[0,1]$ is an integer-valued measure on Borel sets of $[0,1]$.  It may be viewed as a collection of points, which represent the atoms of the measure.  A point process configuration on $[0,1]$ is a collection of point processes, each on $[0,1]$, and can be viewed as a collection of labelled points on $[0,1].$    
 
A Poisson point process on $[0,1]$ with intensity $\alpha$ is a random integer-valued measure which satisfies two properties:  For any Borel subset $ E \subset[0,1]$ with Borel measure $\lambda$, the number of atoms of the point process in $E$ is given by $\Poi(\alpha\lambda)$, and for any disjoint Borel subsets of $[0,1]$ the number of atoms in each are independent.  Conditioned on the number of atoms in $E$ the location of each of the atoms is independent and uniform in $E$.         

Collections of Poisson point processes can be merged to create a single poisson point process.  Suppose $\xi_0$ is a $\Poi(\alpha_0)$ point process on $[0,1]$ and $\xi_1$ is $\Poi(\alpha_1)$ point process on $[0,1]$ with $\xi_0$ and $\xi_1$ both independent.  Then the union of $\xi_0$ and $\xi_1$ is distributed like a $\Poi(\alpha_0+\alpha_1)$ point process.  The reverse is also true.  Let $\xi'$ be a $\Poi(\alpha_0+\alpha_1)$ point process on $[0,1]$ and label each atom $0$ with probability $\alpha_0/(\alpha_0+\alpha_1)$ and $1$ otherwise.  Let $\xi_0$ denote the point process consisting of the atoms labeled $0$ and $\xi_1$ the point process of the remaining atoms.  Then $\xi_0$ and $\xi_1$ are, respectively, independent Poisson($\alpha_0$) and Poisson($\alpha_1$) point processes on $[0,1]$.  This can be generalized further to $\alpha = \alpha_0 + \cdots + \alpha_{k-1}$.  If $\xi'$ is a Poisson($\alpha$) point process each atom in $\xi'$ is independently labeled such that the label is $i$ with probability $\alpha_i/\alpha$ for $0\leq i < k$, then the collection of atoms labeled $i$ is a Poisson($\alpha_i$) point process and each $\xi_i$ is independent of the rest.  

Let $\xi_1$ and $\xi_2$ be two independent Poisson($\alpha)$ point processes.  For $x\in (0,1)$, define $\xi'_1 = \xi_2\big|_{[0,x)} + \xi_1\big|_{(x,1]}$ to be the point process consisting of the atoms from $\xi_2$ restricted to the interval $[0,x)$ and the atoms from $\xi_1$ restricted to the interval $(x,1]$.  If $x$ is independent of $\xi_1$ and $\xi_2$ then the resulting process $\xi'_1$ is also a Poisson($\alpha$) point process.  

\subsection*{Weak Convergence}
We give a brief definition of the version of local weak convergence that is used to define $\limtree$ and $\limdesc$.  See \cite{AS} or \cite{BS} for a proper discussion of local weak convergence, which is sometimes referred to as Benjamini-Schramm convergence.

Let $G_1, G_2 \cdots$ be a sequence of rooted graphs.  For any rooted graph $H$, the $r$-neighborhood of the root, denoted $H(r)$, is the subgraph of $H$ induced from all vertices that are distance at most $r$ from the root.  The rooted graph $G$ is the \emph{local weak limit} of $G_n$ if for every $r\geq 0$ and every finite graph $H$,
$$\prob[G_n(r) = H] \to \prob[G(r) = H].$$
 
\subsection*{From point process configurations to trees} 

\begin{figure}\centering
\includegraphics[scale=1]{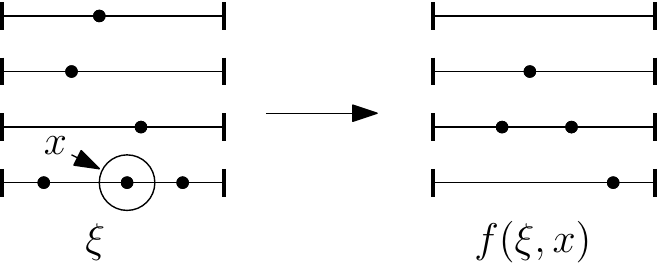}
\caption{The bump map $f(\xi,x)$ where $\xi_4$ is assumed to be empty.  }
\label{bumpmap}	
\end{figure}

Let $\xi=(\xi_k)_{k\geq 0}$ be a point process configuration on $[0,1]^\mathbb{N}$ where each $\xi_k$ is a point process on $[0,1].$  For each atom $x\in \xi_0$ define the bump map $f(\xi,x)= (\xi'_k)_{k\geq 0}$ where
$$\xi'_k = \xi_{k+1} \Big |_{[0,x)} + \xi_k \Big |_{(x,1]}.$$
See Figure \ref{bumpmap} for an illustration of this map.
Given a point process configuration, $\xi$, the bump map allows us to recursively define a tree with root $v_0$ whose vertices are point process configurations.  Define $v_0$ to be the root of the tree with corresponding point process configuration $\xi^{v_0} = \xi$.  Suppose $v$ is a vertex in the tree with corresponding point process configuration given by $\xi^v$.  For each $x\in \xi^v_0$, create a new vertex $v(x)$ in the tree with point process configuration given by the bump map $\xi^{v(x)} = f( \xi^v,x)$.  The newly created vertex $v(x)$ is a considered a child of $v$.  We call this tree the \emph{bump tree} of $\xi$ and denote it by $\gamma(\xi).$  For fixed $r\geq 0$ let $\gamma_r( \xi )$ denote the $r$-neighborhood of the root in $\gamma(\xi).$  Only the atoms in $(\xi_0, \cdots, \xi_{r-1})$ are necessary to determine the structure of the $\gamma_r(\xi)$, so we may write $\gamma_r(\xi) = \gamma_r( \xi_0, \cdots, \xi_{r-1} )$ and assume $\xi_{k} = \emptyset$ for $k\geq r$.   The map $\gamma_r$ is continuous because a slight perturbation of the atoms will not change the relative order of the points in $(\xi_0, \cdots, \xi_r)$.  See Figure \ref{xitree} for an example of a finite neighborhood of the root of the bump tree for a point process configuration.

\begin{figure}\centering
\includegraphics{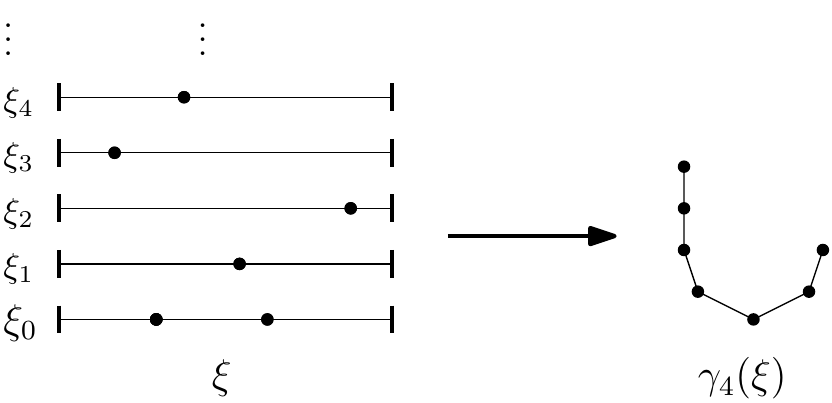}
\caption{A point process collection and corresponding 4-neighborhood of the bump tree.  Note that any configuration of point processes for $\xi_5$ and higher will not affect the structure of the bump tree and thus $\gamma_4(\xi) = \gamma(\xi)$. }
\label{xitree}
\end{figure}


For a permutation $\pi$ of length $n$, we say the index $i$ or the value $\pi(i)$ is \emph{$k$-separated} if $\pi(i) = i+k.$  We define the \emph{separation word} of $\pi$ point-wise by $\word^{\pi}(i):= \pi(i) - i$.  No two permutations have the same separation word.  From this word we can construct a {point process configuration} $(\xi^\pi_k)_{k\geq 0}$ by placing an atom in $\xi^\pi_k$ at position $i/n$ if $i$ is a $k$-separated point in $\pi$.

By Proposition 3.4 in \cite{JSS}, for fixed $r\geq 0$, as $n$ tends to infinity,
$$(\xi^{\pi_n}_0,\cdots , \xi^{\pi_n}_{r-1} ) \longrightarrow_d ( \xi_0, \cdots, \xi_{r-1} )$$
where $\xi_k$ is a $\Poi(1)$ point process on $[0,1]$.  From the arguments of Theorem 3.5 in \cite{JSS}, letting $\xi = ( \xi_k)_{k\geq 0}$, we have $\gamma_r( \xi^{\pi_n}) \to \gamma_r(\xi )$ by continuity of $\gamma_r$ and the Continuous Mapping Theorem [\cite{Billingsley}].  Furthermore, it is seen that $\gamma_r( \xi^{\pi_n} )$ is the same as the $r$-neighborhood of the descendant tree $\desc(\mathbf{\pi}_n)$ with high probability.  Therefore $\limdesc:=\gamma( \xi )$ is the local weak limit of $\desc(\mathbf{\pi}_n)$.

We now can state our main results.  For $\alpha \in (0,1]$, let $\xi = ( \xi_k)_{k\geq 0}$ be a collection of independent $\Poi(\alpha)$ point processes on $[0,1]$ and let $\limdesc_\alpha:=\gamma(\xi)$ be the corresponding bump tree of $\xi$.  Let $D$ denote the number of vertices and $U$ the number of leaves in $\limdesc_\alpha.$  Finally let $\E_\alpha$ and $\P_\alpha$ denote the expectation and probability associated with $\Poi(\alpha)$ point processes.  We now may state our main results.  

\begin{thm}\label{main}
For $0 < \alpha < 1$, $\E_\alpha[D] = (1-\alpha)^{-1}$, and $\E_1[D]$ diverges.
\end{thm}

\begin{thm}\label{leaves}
For $0 < \alpha < 1$, $\E_\alpha[U] = e^{-\alpha}(1-\alpha)^{-1}$, and $\E_1[U]$ diverges.
\end{thm}

\begin{thm}\label{var}
For $\alpha \geq ( \sqrt{5} - 1 ) / 2 $, $\E_\alpha(D^2)$ diverges.
For $\alpha < ( 3 - \sqrt{5}  ) / 2 $, $\E_\alpha(D^2)$ is finite.
\end{thm}

\section{Comparison with Galton-Watson trees}

In this section we compare our results to the well-studied Galton-Watson tree \cite{GaltonWatson, AIHPB_1986__22_2_199_0}.    

A Galton-Watson tree, $\gw$, can be constructed through a simple random process.  Start with a root $v_0$ and a nonnegative integer-valued random variable $X$.  Create $X_{v_0}$ children of $v_0$ where $X_{v_0}$ is distributed as and independent copy of $X$.  For each child, $v$, of $v_0$ repeat this process, where $X_v$ is an independent copy of $X$.  Depending on the distribution of $X$, the resulting tree will have drastically different behavior.    

Fix a nonnegative integer-valued random variable $X$ with finite expectation $0< \E[X] < 1$ and finite second moment $\E[X^2]< \infty.$  Let $Y = |\gw|$.  Let $X$ denote the number of children of the root of $\gw$ and for $1\leq i \leq X$, let $Y^i$ denote the number of vertices in the subtree consisting of the $i$th child and all of its descendants.  Each $Y^{i}$ is distributed identically as an independent copy of $\gw$.  We denote the size of $\gw$ conditioned on $X$ by $(Y|X) = 1+ \sum_{i=1}^X Y^i$.  Taking expectation we have $\E[(Y|X)] = 1+ X\E[Y]$ and thus
$$\E[Y] = \E[\E[(Y|X)]] = 1+ \E[X]\E[Y]$$ and so 
$$\E[Y] = \frac{1}{1-\E[X]}.$$

A similar approach for the second moment gives the equation
$$
\E[Y^2] = 1 + \E[X]\E[Y]+\E[X]\E[Y^2] + \E[X^2 - X]\E[Y]^2, 
$$
which can be simplified to 
\begin{equation}\label{gwvar}
\E[Y^2] = \frac{1}{(1-\E[X])^2} + \frac{\E[X^2]-\E[X]}{(1-\E[X])^3}. 
\end{equation}
Given that $\E[X]<1$ and $\E[X^2]$ is finite, \eqref{gwvar} shows that $\E[Y^2]$ finite.  In particular if $X$ is $\Poi(\alpha)$ then $\E[Y]$ agrees with $\E_\alpha[D]$ from Theorem \ref{main}, while Theorem \ref{var} shows the second moment $\E[Y^2]$ cannot agree with the second moment $\E_\alpha[D^2]$ if $\alpha\geq(\sqrt{5}-1)/2$ since the former is finite while the latter diverges.

The approach used to compute $\E[Y]$ and $\E[Y^2]$ cannot be used to compute $\E_\alpha[D]$ and $\E_\alpha[D^2]$ because the subtrees from the root in $\limdesc_\alpha$ are not independent of each other.

%

\section{Words from point process configurations}

\begin{figure}\centering
\includegraphics{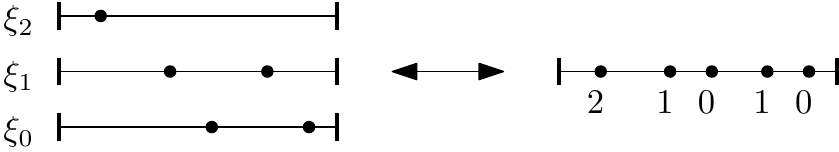}
%
%
%
%
%
	\caption{A collection of point processes corresponding to the word $2 \ 1 \ 0 \ 1 \ 0 $.}
	\label{pp2word}
\end{figure}

For a collection of point processes on $[0,1]$, $\xi=\{\xi_k\}_{k\geq 0}$, let $\bw_r(\xi)$ be the word constructed from the relative order of the atoms in $( \xi_0, \cdots, \xi_{r-1} )$.  For example see Figure \ref{pp2word}.  Assuming that no two atoms of $\xi$ are in the same location, the structure of the $r$-neighborhood of the root in the tree $\gamma_r(\xi)$ can be constructed directly from this word.  Let $\Omega_r$ denote the space of finite words with letters from $\{0,\cdots,r-1\}$.


If $\xi$ is a $\Poi(\alpha)$ point process configuration, this induces a probability measure $\P_{\alpha,r}$ on $\Omega_r$ for every $r\geq 0$.  The following lemma describes this distribution.
  
\begin{lemma} \label{size}
	Let $\xi$ be a $\Poi(\alpha)$ point process configuration and $W=\bw_r(\xi)$ the word given by the relative order of the first $r$ point processes of $\xi$.  Let $w$ denote a fixed word of length $n$ in $\Omega_r$. Then
	\begin{equation}\label{wordlengthn}
\P_{\alpha,r}( |W| = n ) = \frac{1}{n!}e^{-\alpha r}\alpha^nr^n		
	\end{equation} and 
	\begin{equation}\label{wordisw}
	\P_{\alpha,r}( W = w ) = \frac{1}{n!}e^{-\alpha r}\alpha^n.	
	\end{equation}

\end{lemma}

\begin{proof}

Construct the $r$ independent $\Poi(\alpha)$ point processes from a single $\Poi(r \alpha)$ point process by labeling each atom independently from $\{0,\cdots, r-1\}$, choosing the label uniformly at random.  The probability that $|W| = n$ is precisely the probability that a $\Poi(r \alpha)$ point process has $n$ atoms in $[0,1]$, the right hand side of \eqref{wordlengthn}.  As the labeling is independent for each atom, each of the $r^n$ possible labelings is equally likely, so the probability that $W = w$ for a fixed $w$ of length $n$ is computed by dividing the right hand side of \eqref{wordlengthn} by $r^n$, giving \eqref{wordisw}.
\end{proof}

For $W\in \Omega_r$ of length $n$ we write $W = W_1\cdots W_n$ in one line notation. For a fixed subset of indices $A = (i_1,\cdots, i_j)$ let $W_A = W_{i_1}\cdots W_{i_j}$.  We may refine Lemma \ref{size} even further.   

\begin{lemma} \label{specific}

Let $u = u_1 \cdots u_j$ be a word in $\Omega_r$.  Let $W \in \Omega_r$, and $A = (i_1,\cdots, i_j)$ be a set of indices such that $1\leq i_1 < \cdots < i_j \leq n$.  Then, 

$$\P_{\alpha,r}\left( \{W_A = u\}  \cap \{|W| = n\} \right ) =  \frac{1}{n!}e^{-\alpha r} \alpha^n r^{n-j} .$$
	
\end{lemma}

\begin{proof}

Conditioned on $|W|=n$, the labels of the atoms indexed by $A$ are chosen independently so $$\P_{\alpha,r}(W_A=u| |W |= n) = r^{-j}$$ and the statement follows.   
\end{proof}

The tree $\gamma_r(\xi)$ with word $\bw_r(\xi)$ will agree up to a relabeling of the vertices of the tree $\gamma_r(\xi')$ if $\bw_r(\xi) = \bw_r(\xi').$  A vertex in the tree corresponds to bumping a particular set of atoms in a particular order.  Therefore the measure $\P_{\alpha,r}$ on words in $\Omega_r$ is exactly the measure we need to understand the $\gamma_r(\xi)$.  

We can translate our language of bumping atoms in $\xi$ to bumping letters in words.  Let $W\in \Omega_r$. For each $0 \in W$, we construct a new word by removing the chosen $0$ and reducing every letter to the left of it by $1.$  We say the index of this letter $0$ is \emph{bumped} and indices less than the bumped index are \emph{shifted}.  The set of indices of the $0$s in a word are called the \emph{bumpable} indices.  The set of words that can be constructed by bumping a single $0$ in $W$ are called the children of $W$ and denoted $\mathcal{C}(W).$  For example the word $2 \ 1 \ 0 \ 1 \ 0$ has has two children, $1 \ 0 \  \inert \ 1  \ 0$ and $1 \ 0 \ \inert \ 0 \ \inert$, where $\inert$ is used to indicate bumped indices or indices shifted below zero.  Once the letter at an index becomes $\inert$ in a word it can never become $0$ in one of its descendants.  We construct a rooted tree, denoted $\wtree(W)$, following a process that mirrors our construction of $\gamma(\xi)$ for point process configurations.  We let $\wtree_j(W)$ denote the $j$-neighborhood of the root in $\wtree(W)$.   

We may omit the $\inert$ symbol in the labeling of the tree.  The $\inert$ symbol is used to emphasize that the set of indices is the same for each word in the same tree.  See Figure \ref{wordtreefig} for the rooted tree in $\Omega_3$ associated with the word $2 \ 1 \ 0 \ 1 \ 0$.  The sequence of indices that are bumped to reach the vertex $v$ in $\gamma(W)$ is called the bumping sequence of $v$.  

\begin{figure}\centering
\begin{tikzpicture}[
every fit/.style={ellipse,draw,inner sep=2pt},
grow'=up,
leaf/.style={draw, fill, circle, minimum size=4pt, inner sep =1pt},
level 1/.append style={sibling distance=50mm, level distance =12mm},
level 2/.append style={sibling distance=20mm, level distance =10mm},
level 3/.append style={sibling distance=10mm, level distance = 9mm},
level 4/.append style={sibling distance=5mm, level distance = 8mm},
]
\node[draw] (root){$2 \ 1\ 0 \ 1\ 0$} 
	child{ node [draw] {$1 \ 0 \ 1\ 0$ } 
  		child{node [draw] {$ 0 \ 1 \ 0$}
  			child{node [draw]{$1 \ 0$}
  				child{node [draw] {$0$}
  					child{node [draw] {$\emptyset$}
  					}
  				}
  			}
  			child{node [draw] {$0$}
  				child{node [draw] {$\emptyset$}
  				}
  			}
  		}
  		child{node [draw] {$ 0 \ 0$}
  			child{node [draw] {$ 0 $}
  				child{node [draw] {$\emptyset$}
  				}
  			}
  			child{node [draw] {$\emptyset$}
  			}
  		}
  	}
  	child{node [draw] {$ 1 \ 0\ 0$}
  		child{node [draw] {$ 0 \ 0$}
  			child{node [draw] {$ 0 $}
  				child{node [draw] {$\emptyset$}
  				}
  			}
  			child{node [draw] {$\emptyset$}
  			}
  		}
  		child{node [draw] {$0$}
  			child{node [draw] {$\emptyset$}
  			}
  		}
  	}
;	
\end{tikzpicture}
\caption{The tree, $\wtree(\bw)$, for the root word $\bw = 2 \ 1 \ 0 \ 1 \ 0$}
\label{wordtreefig}
\end{figure}
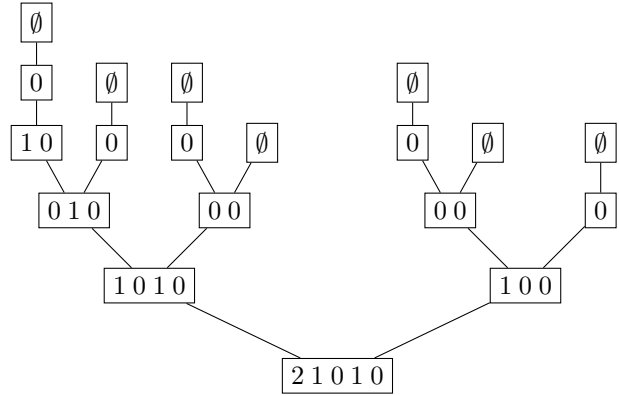

For $j \geq 1$ and every vertex $v \in \wtree_j(W)\backslash \wtree_{j-1}(W)$ there is a corresponding set of $j$ atoms that must be bumped in a particular order to reach $v$.  This sequence of atoms induces an ordered set of indices $A = \{ a_1< \cdots <a_j \}$ and permutation, $\sigma$, of length $j$ such that $v$ is obtained by bumping the atoms at the indices in order $\{a_{\sigma_{1}},\cdots, a_{\sigma_{j}}\}$ where each of the indices must be $0$ when they are bumped.  We say the set of indices $A$ reaches $v$ by the order $\sigma$.  Since $\wtree(W)$ is a tree, any such $v$ is reachable by a unique pair $(A,\sigma)$.  

For a set of indices $A = \{ a_1< \cdots <a_j \}$, we say $A$ is \emph{complete} in $W$ if there exists an order $\sigma \in \mathfrak{S}_j$ and a sequence of words $W=W^0, \cdots, W^j$ such that for $1\leq i \leq j$, $W^{i} \in \mathcal{C}(W^{i-1})$ is obtained by bumping the index $a_{\sigma_i}$ in $W^i$.  Whether or not $A$ is complete in $W$ is independent of the letters not in $A$.  The following lemma gives conditions on when $A$ is complete in $W$.

\begin{lemma}\label{unicomp}

If $A$ is complete in $W\in \Omega_r$ with $|A| = j$, there is a unique $\sigma \in \mathfrak{S}_j$ such that a vertex in $\wtree(W)$ is reachable by $(A,\sigma)$.  If $r\geq j$, then for each $\sigma \in \mathfrak{S}_j$ there is a unique sequence of values $u=u_1\cdots u_j$ such if $W_A = u$ then there exists a vertex in $\wtree(W)$ that is reachable by $(A,\sigma)$.  

Finally, $A$ is complete with respect to $W$ if and only if $W_{a_i} \leq \min(j-i,r-1)$ for $1\leq i \leq j$.

\end{lemma}

\begin{proof}

Since $A$ is complete in $W$ there is at least one $\sigma\in S_j$ and $v$ in $\wtree(W)$ such that $v$ is reachable by $(A,\sigma)$.  First $a_{\sigma_1}$ is bumpable if and only if $W_{a_{\sigma_1}} = 0.$  In order for $a_{\sigma_{i+1}}$ to be bumpable after bumping $a_{\sigma_{1}}$ up to $a_{\sigma_{i}}$, the label of $a_{\sigma_{i+1}}$ must be $0$, and therefore index must be shifted exactly $W_{a_{\sigma_{i+1}}}$ times by bumping indices larger then $a_{\sigma_{i+1}}.$  For this to occur there must be exactly $W_{a_{\sigma_{i+1}}}$ integers $m$ such that $m<i+1$ and $\sigma_m > \sigma_{i+1}.$  In terms of $\sigma^{-1}$ we have for $1\leq i \leq j$,
 $$W_{a_{i}} = \#\{ i<m \leq j | \sigma^{-1}_{i} > \sigma^{-1}_m\}.$$
 The sequence of values $W_{a_1} \cdots W_{a_{j}}$ is the unique inversion table (\cite{knuth}) for the permutation $\sigma^{-1}$.  No two permutations have the same inversion table and thus $\sigma$ must be unique.  Given a $\sigma\in \mathfrak{S}_j$, if $W_A$ is the inversion table for $\sigma^{-1}$ then $A$ will be complete with respect to $W$.
 
 Finally we have that $W_{a_1}\cdots W_{a_j}$ is an inversion table if and only if $W_{a_{i}} \leq j-i$ for $1\leq i \leq j$.  We also have that $W_{a_i} \leq r-1$ by definition.  
 \end{proof} 

Define the following truncated factorial function:
$$
f_y(x) = \left \{  \begin{array}{ll}  x!,&  x \leq y, \\
y! y^{x-y}, & y < x.  	
 \end{array} \right .$$
Note that $\lim _{y \to \infty} f_y(x) = x!$.

Let $\cwords_r(j)$ denote the set of subwords of length $j$ such  such that $A$ is complete in $W$ if and only if $W_A \in \cwords_r(j)$.  For any $r\geq 0$ and $j\geq 0$, by Lemma \ref{unicomp},
 $$| \cwords_r(j) | = f_r(j)$$
 and for $r\geq j$, this simplifies to
  $$| \cwords_r(j) | = j!.$$

\section{Expectation of $D$ and $U$}

Let $D^{(r)}$ denote the number of vertices in $\gamma_r(\xi)$.  Let $U^{(r)}$ denote the number of leaves in $\gamma_r(\xi)$ that are distance less than $r$ from the root.  Note that a leaf in $\gamma_r(\xi)$ that is distance $r$ from the root may not be a leaf in $\gamma_{r+1}(\xi).$  By Theorem 5.1 in  \cite{JSS}, the longest path to a leaf in $\gamma(\xi)$ is almost surely finite and therefore $\gamma_r(\xi)$ is identical to $\gamma(\xi)$ for large enough $r.$  To compute the expectation of $D$ and $U$ it suffices to compute the expectation of $D^{(r)}$ and $U^{(r)}$ and let $r$ tend to infinity.    

Let $W$ be chosen from $\Omega_r$.  For $j\leq r$ let $D^{(r)}_j = |\wtree_j(W)\backslash \wtree_{j-1}(W)|.$  Similarly let $\leaves_j$ denote the set of leaves in $\wtree_j(W)$, so that for $j\leq r-1$, $U^{(r)}_j=|\leaves_j(W)\backslash\leaves_{j-1}(W)|$, the number of leaves in $\wtree_j(W)$ exactly distance $j$ from the root.  By linearity of expectation $$\E_{\alpha,r}[D^{(r)}] = \sum_{j=0}^{r} \E_{\alpha,r}[D^{(r)}_j]$$ and $$\E_{\alpha,r}[U^{(r)}] = \sum_{j=0}^{r-1} \E_{\alpha,r}[U^{(r)}_j].$$  

For a fixed $j \leq n$, let $\mathcal{A}$ be the set of all subsets of $j$ indices $A\subseteq [n]$. Consider a fixed $A  \in \mathcal{A} $ and a word $u$ of length $j$ with letters less than $r$. If a word $W\in \Omega_r$ has length $n$, there are $r^{n-j}$ possible fillings of the indices in $[n] \setminus A $ and there are $f_r(j)$ ways to fill the indices of $A$ so that $A$ is complete in $W$.  

By Lemma \ref{specific} we have
\begin{equation}\label{compprob}
\P_{\alpha,r} \left(\{ A \text{ is complete in } W\}\cap \{|W|=n\} \right) =  e^{-\alpha r}\alpha^nr^{n-j}f_r(j)/n!.	
\end{equation}

By the one-to-one correspondence with complete indices $A$ in $W$ of size $j$ with vertices in $\wtree(W)$ exactly distance $j$ from the root, the expectation of $D^{(r)}_j$ is 

\begin{equation}\label{Ecompprob}
\E_{\alpha,r}[ D^{(r)}_j\indc_{|W|=n} ] = \sum_{A \in \mathcal{A}} e^{-\alpha r}\alpha^nr^{n-j}f_r(j)/n! = 	e^{-\alpha r}\alpha^nr^{n-j} f_r(j) /( j! (n-j)!).
\end{equation}
For $r \geq j$, 

\begin{equation}
\E_{\alpha,r} [ D^{(r)}_ {j}\indc_{|W|=n} ]  = e^{-\alpha r}\alpha^nr^{n-j}/ (n-j)!,
\end{equation}
and $\E_{\alpha,r} [ D^{(r)}_{j} ]  = \sum_{n \geq j} \E [ D^{(r)}_{j}\indc_{|W|=n} ]$, so
\begin{equation} \label{simpleton}
\E_{\alpha,r} [ D^{(r)}_{j} ] =   \alpha^j e^{-\alpha r}\sum_{n\geq j}\frac{(\alpha r)^{n-j}}{ (n-j)!} = \alpha^j.
\end{equation}

\begin{proof}[of Theorem \ref{main}]
 From \eqref{simpleton}, $\E_{\alpha,r}[D^{(r)}_j] = \alpha^j$ for $j\leq r$ and $\E_{\alpha,r}[D^{(r)}] = \sum_{j=0}^r \alpha^j$.  Then $\lim_{r\to\infty }D^{(r)} = D$ and by Monotone Convergence Theorem
$$\E_\alpha[D]=\lim_{r\to \infty} \E_{\alpha,r}[D^{(r)}] = \lim_{r\to\infty} \frac{1-\alpha^{r+1}}{1-\alpha} = \frac{1}{1-\alpha}.$$
\end{proof}

\subsection*{Expected number of leaves}

For a set of indices $A$ of size $j$ that are complete in $W$, let $X$ denote the word obtained after bumping every index in $A$.  The vertex labelled with $X$ is a leaf if it contains no bump-able indices, that is $X$ has no $0$s.  Let $a_0=0$ and $a_{j+1} = |W|+1.$  For $0\leq i \leq j$, an index $b_i \in (a_{i},a_{i+1})$ is bump-able in $X$ if and only if $W_{b_i} = j-i.$      If $r\leq j$ and $i \leq j-r$, $W_{b_i} <r \leq j-i$ and hence $b_i$ cannot be bump-able.  Otherwise if $i> j-r$, there are $r-1$ choices for $W_{b_i}$ so that $b_i$ is not bump-able.

Let $\ell(r,n,A)$ denote the number words, $w$ of length $n$ in $\Omega_r$ such that $A$ corresponds to a leaf in $\wtree(w)$.  There are $f_r(j)$ possible ways to fill in the indices of $A$.  For $r\leq j$,
\begin{equation}
\ell(r,n,A) = f_r(j)r^{\sum_{i=0}^{j-r}(a_{i+1}-a_i - 1) }(r-1)^{\sum_{i=j-r+1}^{j}(a_{i+1}-a_i - 1) } .	 	
\end{equation}
For $j<r$ this simplifies to \begin{equation}
 	\ell(r,n,A)=j!(r-1)^{n-j}.
 \end{equation}

Thus for $j<r$ we have
\begin{equation}
	\P_{\alpha,r} \left( \Big\{ |W| = n \Big\} \bigcap \Big\{ X \text{ is a leaf} \Big\} \right) =  e^{-\alpha r}\alpha^n(r-1)^{n-j}j!/n!.
\end{equation}
For $j<r$ the expectation of $U^{(r)}_j\indc_{\{|W|=n\}}$ is
\begin{equation}
\E_{\alpha,r} [ U^{(r)}_j\indc_{\{|W|=n}\}] = \sum_{A \in \mathcal{A}} e^{-\alpha r}\alpha^n(r-1)^{n-j}j!/n! = e^{-\alpha r}\alpha^n(r-1)^{n-j}/ (n-j)!.
 \end{equation}

Summing over $n\geq j$ gives
\begin{equation}\label{leafygreens}
	\E_{\alpha,r}[U^{(r)}_j] = e^{-\alpha r}\alpha^j \sum_{n\geq j} (\alpha(r-1))^{n-j}/(n-j)! = e^{- \alpha }\alpha^j.
\end{equation}

\begin{proof}[of Theorem \ref{leaves}]

 From \eqref{leafygreens}, $\E_{\alpha,r}[U^{(r)}_j] = e^{-\alpha}\alpha^j$ for $j< r$ and $\E_{\alpha,r}[U^{(r)}] = \sum_{j=0}^{r-1} e^{-\alpha} \alpha^j$.  Then $\lim_{r\to\infty }U^{(r)} = U$ and by Monotone Convergence Theorem
\begin{equation}
	\E_\alpha[U]=\lim_{r\to \infty} \E_{\alpha,r}[U^{(r)}] = \lim_{r\to\infty} e^{-\alpha}\frac{1-\alpha^r}{1-\alpha} = \frac{e^{-\alpha}}{1-\alpha}.
\end{equation}\end{proof}

\section{Expectation of $D^2$}

For $a,b,c,m \geq 0 $ let $n = a+b+c+m$.  Let $\intersection(a,b,c,m)$ be the set of all ordered pairs of subsets of $[n]$, $(A,B)$, such that $|A \setminus B| = a $, $|B \setminus A| = b$, and $|A \cap B| = c$ and let $\intersection(a,b,c) = \bigcup_m \intersection(a,b,c,m)$.  We denote the set of distinct subwords $u$ on the indices $A\cup B$ such that and both $u_A$ and $u_B$ are complete by $\chi_r(A,B)$.  The size of $\chi_r(A,B)$ is denoted by $x_r(A,B)$ and only depends on the relative order of $A$ and $B$.  Suppose $(A,B) \in \intersection(a,b,c)$.  For both subwords to be complete, each index $a_i \in A \setminus B$ must have letters strictly less than $\min (a+c-i, r)$, each index $b_j \in B \setminus A$ must have letters strictly less than $\min (b+c-j, r)$, and each index $a_i = b_j \in A \cap B$ must have letters strictly less than $ \min( a+c-i , b+c - j, r)$.    Thus
 
 \begin{equation}\label{chitown}
 x_r({A,B}) =  \frac{f_r(a+c) f_r(b+c) }{ \prod_{a_i = b_j} \min ( r,  \max (a+c - i, b+c - j))}. 
 \end{equation}
 
 The following lemma provides uniform bounds of $x_r(A,B)$ for all $(A,B)\in \intersection(a,b,c)$.

\begin{lemma}\label{lemchi}
Fix $a,b,c$ and $r \geq 0$.
For $(A,B) \in \intersection(a,b,c)$, if $a\leq b$, then 
$$f_r(a+c) f_r(b) \leq x_r (A,B) \leq (a+c)! (b+c)! / c!$$
Otherwise if $a>b$, then $$f_r(b+c)f_r(a) \leq x_r(A,B)\leq (a+c)! (b+c)! / c!.$$

\end{lemma}

\begin{proof}

For a fixed $a,b,c$ and $r$, $ x_r({A,B}) $  will reach its minimum value over $\intersection(a,b,c)$ when the product in the denominator is maximized in the right hand side of \eqref{chitown}. 
The denominator of $x_r(A,B)$ is maximized when every index in $A\cap B$ is less than every index in $A\cup B \setminus A\cap B$ so $A\cap B = \{ a_1 =  b_1, \cdots , a_c = b_c\}$. In this case for $a\leq b$ the denominator of the right hand side of \eqref{chitown} is given by 
$$\prod_{i=1}^c \min(r, b + i) = f_r(b+c)/f_r(b)
$$
and $$x_r({A,B}) = f_r(a+c) f_r(b).$$ Otherwise for $a>b$ $$x_r(A,B) = f_r(b+c)f_r(a).$$  

For the other direction $x_r(A,B)$ is maximized when the denominator in the right-hand side of \eqref{chitown} is minimized.  This occurs when every index in $A\cap B$ is greater than every index in $A\cup B \setminus A\cap B$.  In this case, \begin{equation} x_r({A,B}) =  \frac{ f_r(a+c) f_r(b+c) }{ f_r(c)} \leq  \frac{(a+c)! (b+c)! }{ c!}. \qedhere\end{equation}
\end{proof}

These bounds on $x_r(A,B)$ will give us bounds on $\E_\alpha[D^2].$  Let $V_r = 1+\sum_{j=1}^\infty D^{(r)}_j.$  For a fixed set of indices $A\in \mathbb{Z}_+$ let $\indc_{A}(W)$ denote the indicator function that is $1$ if $W_A$ is complete and $0$ if $W_A$ is not complete or $A$ is not a subset of indices of $W$.  Then
\begin{equation*}
V_r = \sum_{A \subset \mathbb{Z}_+} \indc_{A}(W)
\end{equation*}
with $\lim_{r\to\infty} V_r = D$.  We also have
\begin{align*}
V_r^2 &= \sum_{(A,B) \subset \mathbb{Z}_+^2 } \indc_{A}(W)\indc_{B}(W) \\
&= \sum_{a,b,c}\sum_{\intersection(a,b,c)} \indc_A(W)\indc_B(W)\\
& = \sum_{a,b,c,m} \sum_{\intersection(a,b,c,m)} \indc_{A}(W)\indc_{B}(W)\indc_{|W|=a+b+c+m}.
\end{align*}
For a fixed pair $(A,B) \in \intersection(a,b,c,m)$, using Lemma \ref{specific} we have   

\begin{align} 
\E_{\alpha,r}[\indc_{A}(W)&\indc_{B}(W)\indc_{\{|W|=a+b+c+m\}}] \nonumber\\ 
&= \sum_{u\in\chi_r(A,B)}\P_{\alpha,r}\left( \{W_{A\cup B} = u\} \cap \{|W| = a+b+c+m\} \right) \nonumber \\
&= \frac{1}{(a+b+c+m)!}e^{-\alpha r }\alpha^{a+b+c+m}r^{m}x_r(A,B).\label{carlosdanger}
\end{align}

The value of $x_r(A,B)$ depends on $(A,B)$ but the upper and lower bounds from Lemma \ref{lemchi} only depend on $a,b,$ and $c$.  Thus we have bounds of \eqref{carlosdanger} that are uniform for all $(A,B) \in \intersection(a,b,c,m)$.  For each $m$ the size of $\intersection(a,b,c,m)$ is ${ a+b+c+m \choose a, b, c, m } = \frac{(a+b+c+m)!}{a!b!c!m!}$.  Thus
 \begin{multline}\label{annoyed}
	\sum_{\intersection(a,b,c,m)} \E_{\alpha,r}\left[\indc_{A}(W)\indc_{B}(W)\indc_{|W|=a+b+c+m}\right]\\ 
	\geq \frac{\alpha^{a+b+c}}{a!b!c!} f_r(\max(a,b) +c) f_r( \min(a,b)) \frac{1}{m!}(\alpha r)^{m}e^{-\alpha r} 
\end{multline}
Summing over $m\geq 0$ in \eqref{annoyed} gives the lower bound
\begin{equation}
\sum_{\intersection(a,b,c)} \E_{\alpha,r}\left[\indc_{A}(W)\indc_{B}(W)\right] \geq \frac{\alpha^{a+b+c}}{a!b!c!} f_r(\min(a,b) +c) f_r( \max(a,b)).
\end{equation}
Similarly for the upper bound we have 
\begin{equation}
\sum_{\intersection(a,b,c)} \E_{\alpha,r}\left[\indc_{A}(W)\indc_{B}(W)\right]\leq \alpha^{a+b+c}{a+c \choose c}{b+c \choose c}.
\end{equation}

\begin{proof}[of Theorem \ref{var}]
	
In this section we make repeated use of the identity $$\sum_{n\geq 0} {n+k\choose n} x^n = \frac{1}{(1-x)^{k+1}}.$$  See \cite{gfunc} for a variety of similar identities.  
	
	By Fatou's Lemma $\lim_{r\to \infty} \E_{\alpha,r}[V_r^2] \leq \E_\alpha[\lim_{r\to\infty} V_r^2] = \E_\alpha[D^2]$ so

\begin{align}
\lim_{r\to\infty}\sum_{a < b,c}\frac{\alpha^{a+b+c}}{a!b!c!} f_r(\min(a,b) +c) f_r( \max(a,b)) &\leq \sum_{0\leq a < b,0\leq c }{a+c \choose a}\alpha^{a+b+c} \label{vellhart} \\ 
&\leq \E_\alpha[\lim_{r\to\infty}V_r^2]  \nonumber\\
&= \E_\alpha[D^2]. \nonumber
\end{align}

The right hand side of \eqref{vellhart} can be simplified further.  Suppose $1/2< \alpha < 1.$  Then 
\begin{align}
	\sum_{0\leq a < b,0\leq c }{a+c \choose c}\alpha^{a+b+c} &= \sum_{0\leq a < b} \frac{\alpha^{b}}{1-\alpha}\left(\frac{\alpha}{(1-\alpha)}\right)^a \label{jabba1}\\
	&=\frac{1}{2\alpha-1}\sum_{b> 0} \alpha^b\left(\left(\frac{\alpha}{1-\alpha}\right)^{b} - 1\right)\label{jabba2}\\
	&=\frac{1}{2\alpha-1}\sum_{b>0} \left(\frac{\alpha^{2}}{1-\alpha}\right)^b - \alpha^b. \label{jabba3}
\end{align}

There is an issue when $\alpha = 1/2$ in \eqref{jabba2} and \eqref{jabba3}.  But in this case $\frac{ \alpha}{ 1-\alpha} = 1$ in \eqref{jabba1}, so \eqref{jabba2} becomes $\sum_{b\geq 0} \frac{b \alpha^b}{1-\alpha},$ which is finite.  
Otherwise \eqref{jabba3} diverges precisely when $\alpha^2/(1-\alpha) \geq 1$ which occurs if $({\sqrt{5} - 1})/{2} \leq \alpha <1 .$
For the other direction we have 

\begin{align}
\E_\alpha[D^2] & = \E_\alpha[\lim_{r\to \infty}V_r^2] \nonumber\\
& \leq \sum_{a,b,c\geq 0} {a+c \choose c}{b+c \choose c} \alpha^{a+b+c}\nonumber\\
 & = \sum_{b,c \geq 0} {b+c \choose c} \frac{\alpha^{b+c}}{(1-\alpha)^{c+1}}\nonumber\\
 & =\frac{1}{(1-\alpha)^2}\sum_{c\geq 0} \left(\frac{\alpha}{(1-\alpha)^{2}}\right)^c	\label{lastly}
\end{align}
  The last line \eqref{lastly} converges when ${\alpha}/{(1-\alpha)^2} < 1,$ which occurs when $0 < \alpha < (3-\sqrt{5})/2.$
\end{proof}

As $\alpha$ increases from $(3-\sqrt{5})/2$ to $({\sqrt{5}-1})/{2}$ a phase transition occurs where $\E_\alpha[D^2]$ becomes infinite.  With a more precise analysis of the size of $x_r(A,B)$ that depends more closely on the relative order of $A$ and $B$, one might be able to obtain the exact location where this phase transition occurs.

\section*{Acknowledgements}

We wish to express thanks to Tobias Johnson and Anne Schilling for useful discussions.


\bibliographystyle{abbrvnat}
\bibliography{shuffle_forest}

\end{document}